\documentclass[a4paper,11pt,reqno]{amsart}
\usepackage{amsmath,amssymb,url,stmaryrd}
\usepackage[T1]{fontenc}
\usepackage{lmodern}
\usepackage[top=30truemm,bottom=30truemm,left=20truemm,right=20truemm]{geometry}
\usepackage{tikz}
\usetikzlibrary{positioning}
\usetikzlibrary{arrows.meta}
\usetikzlibrary{calc}
\usepackage{color}
\usepackage{enumitem}
\usepackage{hyperref}

\newenvironment{enumerate2}{\begin{enumerate}[label=\textup{(\alph*)}]}
{\end{enumerate}}

\date{\today}

\newtheorem{Thm}{Theorem}[section]
\newtheorem{Cor}[Thm]{Corollary}
\newtheorem{Lem}[Thm]{Lemma}
\newtheorem{Prop}[Thm]{Proposition}
\newtheorem{Conj}[Thm]{Conjecture}
\newtheorem{Def-Lem}[Thm]{Definition-Lemma}
\newtheorem{Def-Prop}[Thm]{Definition-Proposition}

\theoremstyle{definition}
\newtheorem{Def}[Thm]{Definition}

\newtheorem{Ex}[Thm]{Example}

\newtheorem{Ques}[Thm]{Question}

\newcommand{\calC}{\mathcal{C}}
\newcommand{\calD}{\mathcal{D}}
\newcommand{\calF}{\mathcal{F}}

\newcommand{\calS}{\mathcal{S}}
\newcommand{\calT}{\mathcal{T}}

\newcommand{\calW}{\mathcal{W}}

\newcommand{\calZ}{\mathcal{Z}}

\newcommand{\ovcalF}{\overline{\calF}}
\newcommand{\ovcalT}{\overline{\calT}}

\newcommand{\sfD}{\mathsf{D}}
\newcommand{\sfK}{\mathsf{K}}

\newcommand{\rmb}{\mathrm{b}}

\newcommand{\R}{\mathbb{R}}
\newcommand{\Z}{\mathbb{Z}}

\DeclareMathOperator{\Hom}{\mathsf{Hom}}
\DeclareMathOperator{\End}{\mathsf{End}}

\DeclareMathOperator{\Ker}{\mathsf{Ker}}

\DeclareMathOperator{\module}{\mathsf{mod}} \renewcommand{\mod}{\module}
\DeclareMathOperator{\rep}{\mathsf{rep}}

\DeclareMathOperator{\proj}{\mathsf{proj}}

\DeclareMathOperator{\inj}{\mathsf{inj}}

\DeclareMathOperator{\add}{\mathsf{add}}

\DeclareMathOperator{\Filt}{\mathsf{Filt}}
\DeclareMathOperator{\Fac}{\mathsf{Fac}}
\DeclareMathOperator{\Sub}{\mathsf{Sub}}

\DeclareMathOperator{\tors}{\mathsf{tors}}
\DeclareMathOperator{\ftors}{\mathsf{f-tors}}
\DeclareMathOperator{\torf}{\mathsf{torf}}

\DeclareMathOperator{\sbrick}{\mathsf{sbrick}}
\DeclareMathOperator{\wide}{\mathsf{wide}}

\DeclareMathOperator{\twosilt}{2\mathsf{-silt}}

\DeclareMathOperator{\twopsilt}{2\mathsf{-psilt}}

\DeclareMathOperator{\ovtors}{\overline{\mathsf{tors}}}

\DeclareMathOperator{\sfT}{\mathsf{T}}
\DeclareMathOperator{\sfF}{\mathsf{F}}

\DeclareMathOperator{\sfdim}{\mathsf{dim}} \renewcommand{\dim}{\sfdim}

\DeclareMathOperator{\Irr}{\mathsf{Irr}}

\newcommand{\cw}{\mathrm{cw}}

\DeclareMathOperator{\cone}{\mathsf{cone}}

\renewcommand{\epsilon}{\varepsilon}
\renewcommand{\Gamma}{\varGamma}
\renewcommand{\Lambda}{\varLambda}
\renewcommand{\Phi}{\varPhi}
\renewcommand{\phi}{\varphi}

\numberwithin{equation}{section}

\begin{document}
\title[Bicompact torsion classes and
conjectures on brick infinite algebras]
{Bicompact torsion classes and \\
conjectures on brick infinite algebras}

\author{Sota Asai} 
\address{Sota Asai: Graduate School of Mathematical Sciences,
the University of Tokyo,  
3-8-1 Komaba, Meguro-ku, Tokyo-to, 153-8914, Japan}
\email{sotaasai@g.ecc.u-tokyo.ac.jp}

\begin{abstract}
In the category $\mod A$ of finitely generated modules 
over a finite dimensional algebra $A$, 
a torsion class $\calT \subset \mod A$ is said to be \emph{compact}
if there exists a module $M \in \mod A$ such that $\calT$ 
is the smallest torsion class containing $M$.
If a torsion class satisfies this and the dual condition,
then we call it a \emph{bicompact} torsion class.
We conjecture that bicompact torsion classes are precisely 
functorially finite torsion classes,
and prove it for hereditary algebras and also for semistable torsion classes.
This gives that Demonet's Lattice Point Conjecture implies 
Enomoto's Infinite Semibrick Conjecture,
both of which are important conjectures on brick infiniteness.

\end{abstract}

\maketitle 

\section{Introduction}

For a finite dimensional algebra $A$ over a field $K$,
classification of nice full subcategories 
of the category $\mod A$ of finitely generated modules
is one of the most important problems.
Especially, researchers actively have investigated
\emph{torsion classes} $\calT \subset \mod A$,
that is, full subcategories closed under taking factor modules and extensions.
For each torsion class $\calT$, we uniquely have
the corresponding \emph{torsion-free class} $\calF$
which forms a \emph{torsion pair} $(\calT,\calF)$ in $\mod A$.

We write $\tors A$ for the set of torsion classes.
In general, it is very hard to classify all $\calT \in \tors A$,
so some extra assumption for torsion classes is often assumed.
The most fundamental subfamily of torsion classes is
\emph{functorially finite} torsion classes.
By \cite{Smalo}, $\calT \in \tors A$ is functorially finite
if and only if 
there exists some module $M \in \mod A$ such that $\calT=\Fac M$,
where $\Fac M$ is the full subcategory of factor modules of 
$M^{\oplus s}$ ($s \ge 0$).

There are several kinds of generalizations of 
functorially finite torsion classes.
In this paper, we focus on the following families of torsion classes:
\begin{enumerate}
\item
compact (cocompact, bicompact) torsion classes;
\item
semistable torsion classes; and
\item
widely generated torsion classes.
\end{enumerate}

We first give the definitions of compact torsion classes and related notions.
For each $M \in \mod A$, we write $\sfT(M)$ (resp.~$\sfF(M)$) for
the smallest torsion class (resp.~torsion-free class) 
containing $M$.

\begin{Def}[Definition \ref{Def_bicpt}]
Let $\calT \in \tors A$.
Then $\calT$ is said to be \emph{compact}
if there exists $M \in \mod A$ such that $\calT=\sfT(M)$,
and said to be \emph{cocompact}
if there exists $N \in \mod A$ such that $\calT^\perp=\sfF(N)$,
and said to be \emph{bicompact}
if $\calT$ is both compact and cocompact.
\end{Def}

We remark that the compactness is originally a notion in lattice theory,
and the definition above is a characterization by \cite{DIRRT}
for the lattice $\tors A$ of torsion classes.

Functorially finite classes are always bicompact \cite{Smalo}.
We expect that the converse also holds.

\begin{Conj}\label{Conj_bicpt}
Let $A$ be a finite dimensional algebra over a field $K$, 
and $\calT \in \tors A$.
Then $\calT$ is bicompact if and only if $\calT$ is functorially finite.
\end{Conj}

The following is our first main theorem of this paper,
which gives a partial positive answer.

\begin{Thm}[Theorem \ref{Thm_hered}]\label{Thm_hered_intro}
Assume that $A$ is a finite dimensional hereditary algebra 
over an algebraically closed field $K$.
Then Conjecture \ref{Conj_bicpt} holds.
\end{Thm}

To show this, we establish more general thoery by using
\emph{semistable torsion classes} \cite{BKT}.
We consider the category $\proj A$ of finitely generated projective modules,
and its real Grothendieck group $K_0(\proj A)_\R$.
Each $\theta \in K_0(\proj A)_\R$ is identified with
an $\R$-linear form $\theta \colon K_0(\mod A)_\R \to \R$
via the \emph{Euler bilinear form}
$K_0(\proj A)_\R \times K_0(\mod A)_\R \to \R$.
This gives rise to 
two torsion classes $\calT_\theta \subset \ovcalT_\theta$ defined by
\begin{align*}
\calT_\theta&:=\{M \in \mod A \mid
\text{$\theta(N)>0$ for any nonzero factor module $N \ne 0$ of $M$}\},\\
\ovcalT_\theta&:=\{M \in \mod A \mid
\text{$\theta(N) \ge 0$ for any factor module $N$ of $M$}\}.
\end{align*}
We remark that $\calW_\theta:=\ovcalT_\theta \cap (\calT_\theta)^\perp$
is nothing but the \emph{$\theta$-semistable subcategory}
introduced in \cite{King}.

Functorially finite torsion classes are always realized
as semistable torsion classes.
Thanks to silting theory \cite{KV,AIR},
every functorially finite torsion class $\calT$
is of the form $\calT=\Fac H^0(T)$,
where $T$ is a 2-term silting complex 
in the perfect derived category $\sfK^\rmb(\proj A)$.
Writing $T=\bigoplus_{i=1}^n T_i$ with $T_i$ indecomposable,
we have the \emph{silting cone} 
$C^\circ(T):=\sum_{i=1}^n \R_{>0}[T_i]$ in $K_0(\proj A)_\R$.
Then any $\theta \in C^\circ(T)$ satisfies
$\calT_\theta=\ovcalT_\theta=\Fac H^0(T)$ \cite{BST,Yurikusa}.

The closures $C(T)$ of silting cones and its faces 
form a non-singular fan in $K_0(\proj A)_\R$ 
called the \emph{$g$-fan} \cite{AHIKM}.
As explained later, this fan is not necessarily complete.
The \emph{wall-chamber structure} by \cite{BST}
and the \emph{TF equivalence} by \cite{A-wc} on $K_0(\proj A)_\R$ can be
considered as a completion of the $g$-fan.
See also \cite{KT} for these notions.

Our second main result below gives an affirmative answer
to Conjecture \ref{Conj_bicpt}
for semistable torsion classes $\calT_\theta$ and $\ovcalT_\theta$.
Here, $\theta \in K_0(\proj A)_\R$ is said to be \emph{rigid}
if $\theta$ is in the support of the $g$-fan.

\begin{Thm}[Theorems \ref{Thm_sst_bicpt} and \ref{Thm_ovT_Z_bicpt}]
\label{Thm_bicpt_intro}
Let $A$ be a finite dimensional algebra over a field $K$.
For each $\theta \in K_0(\proj A)_\R$, 
the following conditions are equivalent.
\begin{enumerate2}
\item
The element $\theta$ is rigid.
\item
The torsion class $\calT_\theta$ is bicompact.
\item
The torsion class $\calT_\theta$ is compact.
\item
The torsion class $\calT_\theta$ is functorially finite.
\item
The torsion class $\ovcalT_\theta$ is bicompact.
\item
The torsion class $\ovcalT_\theta$ is cocompact.
\item
The torsion class $\ovcalT_\theta$ is functorially finite.
\end{enumerate2}
If $K$ is algebraically closed and $\theta \in K_0(\proj A)$, 
then they are equivalent also to the following.
\begin{enumerate2}
\item[\textup{(h)}]
The torsion class $\calT_\theta$ is cocompact.
\item[\textup{(i)}]
The torsion class $\ovcalT_\theta$ is compact.
\end{enumerate2} 
\end{Thm}

We say that $\calT \in \tors A$ is a \emph{numerically disjoint} torsion class
if there exist no $X \in \calT$ and $Y \in \calT^\perp$
such that $[X]=[Y] \ne 0$ in $K_0(\mod A)$.
As a consequence of Theorem \ref{Thm_bicpt_intro},
we get the following property.

\begin{Thm}[Theorem \ref{Thm_bicpt_num_dis}]
\label{Thm_bicpt_num_dis_intro}
Let $A$ be a finite dimensional algebra over a field $K$, 
and $\calT \in \tors A$.
Then $\calT$ is functorially finite if and only if
$\calT$ is bicompact and numerically disjoint.
\end{Thm}

Then Theorem \ref{Thm_hered_intro} follows 
from Theorem \ref{Thm_bicpt_num_dis_intro},
together with some arguments on the representation varieties $\rep(A,d)$; 
see Propositions \ref{Prop_either_empty} and \ref{Prop_dis_irr_comp}.

As another application of Theorem \ref{Thm_bicpt_intro},
we show a certain implication between open questions on brick infiniteness.

For $M \in \mod A$,
if its endomorphism algebra $\End_A(M)$ is a division $K$-algebra,
then $M$ is called a \emph{brick}.
A \emph{semibrick} is defined as a set of isoclasses of bricks
which are pairwisely Hom-orthogonal.
By \cite{Ringel1},
semibricks are in bijection with \emph{wide subcategories} of $\mod A$,
where each semibrick $\calS$ is sent to its filtration closure
$\Filt \calS \subset \mod A$.

Thus a torsion class $\calT \in \tors A$ is said to be \emph{widely generated}
if there exists a semibrick $\calS$ such that $\calT=\sfT(\calS)$ 
as in \cite{AP}.
In this case, such a semibrick is unique by \cite{IT,MS}.
There are more interesting results on 
the relationship between semibricks and torsion classes;
see \cite{BCZ,BH,DIRRT,Hanson,Jasso,Ringel2,Shimpi} for example.

A finite dimensional algebra $A$ is said to be \emph{brick infinite}
if $\mod A$ has infinitely many isoclasses of bricks.
Brick (in)finiteness of algebras are characterized 
in many ways by using torsion classes and 2-term silting complexes;
for example, see Proposition \ref{Prop_brick_infin} 
and Corollary \ref{Cor_brick_infin_tors}
in this paper cited from \cite{AS,A-wc,DIJ,Ringel2,Sentieri1}.
These characterizations are strengthened 
in Corollary \ref{Cor_brick_infin_tors_sst} 
by Theorem \ref{Thm_bicpt_intro}.
For more developments, see also \cite{BPPW,NasrIsfahani,Sentieri2}.

Moreover, there are many conjectures 
to characterize brick infiniteness \cite{Mousavand,MP,STV,Pfeifer}.
In this paper, we consider the following two conjectures.

The first one was proposed by Enomoto \cite{Enomoto}.
A semibrick consisting of infinitely many isoclasses of bricks
is called an \emph{infinite semibrick}.
If there exists an infinite semibrick in $\mod A$,
then $A$ is clearly brick infinite.
In \cite[Conjecture 5.12]{Enomoto},
Enomoto conjectured that the converse also holds.

\begin{Conj}[Conjecture \ref{Conj_Enomoto}, 
Enomoto's Infinite Semibrick Conjecture]\label{Conj_Enomoto_intro}
Let $A$ be a finite dimensional algebra over a field $K$. 
If $A$ is brick infinite, then there exists an infinite semibrick in $\mod A$.
\end{Conj}

Enomoto himself gave a partial answer in terms of \emph{monobricks};
see \cite[Theorem 5.11]{Enomoto}.

The other conjecture was suggested by Demonet \cite{Demonet}.
In \cite{A-wc}, we proved that $A$ is brick infinite if and only if 
there exists $\theta \in K_0(\proj A)_\R$ which is not rigid.
Demonet questioned whether the following strong version holds
in \cite[Question 3.49]{Demonet}, and it is now recognized as a conjecture.

\begin{Conj}[Conjecture \ref{Conj_Demonet}, 
Demonet's Lattice Point Conjecture]
\label{Conj_Demonet_intro}
Let $A$ be a finite dimensional algebra over a field $K$. 
If $A$ is brick infinite, then there exists a lattice point 
$\theta \in K_0(\proj A)$ which is not rigid.
\end{Conj}

One of the reasons why lattice points are important is that
the \emph{presentation space} $\Hom(\theta)$
for $\theta \in K_0(\proj A)$ is available.
There are many studies including 
\cite{AsI,DF,Fei,HY,Plamondon,PYK} in this direction,
and such ideas play an important role also in this paper
when we prove the latter part of Theorem \ref{Thm_bicpt_intro}. 

Another advantage of lattice points is that
the torsion class $\ovcalT_\theta$ is widely generated 
if $\theta \in K_0(\proj A)$ \cite{A-wc}.
By this and Theorem \ref{Thm_bicpt_intro},
we prove the following relationship between the two conjectures.

\begin{Thm}[Theorem \ref{Thm_Demonet_Enomoto}]
\label{Thm_Demonet_Enomoto_intro}
Assume that $A$ is a finite dimensional algebra 
over an algebraically closed field $K$.
If $A$ satisfies Demonet's Lattice Point Conjecture \ref{Conj_Demonet_intro},
then $A$ satisfies 
Enomoto's Infinite Semibrick Conjecture \ref{Conj_Enomoto_intro}.
\end{Thm}

After the first version of this paper appeared on arXiv, 
Mousavand-Paquette \cite{MP2} gave 
another proof for Theorem \ref{Thm_Demonet_Enomoto_intro},
via the Second brick-Brauer-Thrall Conjecture.

The organization of this paper is as follows.
In Section \ref{Sec_Pre}, we define bicompact torsion classes,
and recall some basic properties on functorially finite torsion classes
and semistable torsion classes.
Section \ref{Sec_semistable} is devoted to 
the proof of Theorem \ref{Thm_bicpt_intro}.
In Section \ref{Sec_hered}, 
we show Theorem \ref{Thm_bicpt_num_dis_intro}
by using numerically disjoint torsion classes,
and deduce Theorem \ref{Thm_hered_intro}.
In Section \ref{Sec_Conj}, 
we prove Theorem \ref{Thm_Demonet_Enomoto_intro},
and then discuss characterizations of brick infiniteness
from the point of view of Theorem \ref{Thm_bicpt_intro}
and Corollary \ref{Conj_Demonet_intro}.

\subsection*{Convention}
In this paper, $K$ is a field,
and $A$ is a finite dimensional $K$-algebra.
We write $\mod A$ for the category of finitely generated right $A$-modules,
and $\proj A$ 
for the category of finitely generated projective right $A$-modules.

\subsection*{Acknowledgement}
The author thanks to Esha Gupta, Osamu Iyama, Kaveh Mousavand, 
Charles Paquette, David Pauksztello and
Claus Michael Ringel for useful comments.
This work was supported by JSPS KAKENHI Grant Number JP23K12957. 

\section{Preliminary}\label{Sec_Pre}

For a full subcategory $\calC \subset \mod A$, we set
\begin{align*}
{^\perp \calC}:=\{X \in \mod A \mid \Hom_A(X,\calC)=0\}, \quad
{\calC^\perp}:=\{X \in \mod A \mid \Hom_A(\calC,X)=0\}.
\end{align*}

Let $\calT,\calF$ be full subcategories of $\mod A$.
Then the pair $(\calT,\calF)$ is called a \emph{torsion pair}
if $\calF=\calT^\perp$ and $\calT={^\perp \calF}$.
A full subcategory $\calT \subset \mod A$ is called 
a \emph{torsion class} in $\mod A$ 
if there exists a full subcategory $\calF \subset \mod A$
such that $(\calT,\calF)$ is a torsion pair
(in this case, $\calF=\calT^\perp$).
This is equivalent to that $\calT$ is closed under taking
factor modules and extensions in $\mod A$.
We write $\tors A$ for the set of all torsion classes in $\mod A$.
Dually, \emph{torsion-free classes} are defined,
and $\torf A$ denotes the set of all torsion-free classes in $\mod A$.

Let $\calC \subset \mod A$ be a full subcategory.
We write $\sfT(\calC)$ (resp.~$\sfF(\calC)$) for the smallest
torsion class (resp.~torsion-free class) containing $\calC$.
Under this notation, we have two torsion pairs
$(\sfT(\calC),\calC^\perp)$ and $({^\perp \calC},\sfF(\calC))$.
In the case $\calC=\{M\}$, we simply write $\sfT(M)$ and $\sfF(M)$.
We set
\begin{align*}
\Filt \calC&=\{ X \in \mod A \mid 
\text{there exists $0=X_0 \subset X_1 \subset \cdots \subset X_l=X$
with $X_i/X_{i-1} \in \calC$} \}, \\
\Fac \calC&=\{ X \in \mod A \mid 
\text{there exists a surjection $C^{\oplus s} \to X$ 
with $C \in \calC$ and $s \ge 0$}\}, \\
\Sub \calC&=\{ X \in \mod A \mid 
\text{there exists an injection $X \to C^{\oplus s}$ 
with $C \in \calC$ and $s \ge 0$}\}.
\end{align*}
Then we have $\sfT(\calC)=\Filt(\Fac \calC)$ and 
$\sfF(\calC)=\Filt(\Sub \calC)$;
see \cite[Lemma 3.1]{MS}.

We are mainly interested in the case $\calC=\{M\}$ for $M \in \mod A$.
The following notions are the main topics of this paper.

\begin{Def}\label{Def_bicpt}\cite[Proposition 3.2]{DIRRT}
Let $\calT \in \tors A$.
Then $\calT$ is said to be \emph{compact}
if there exists $M \in \mod A$ such that $\calT=\sfT(M)$,
and said to be \emph{cocompact}
if there exists $M \in \mod A$ such that $\calT^\perp=\sfF(M)$,
and said to be \emph{bicompact}
if $\calT$ is both compact and cocompact.
\end{Def}

We remark that the original definition of compactness
comes from lattice theory as in \cite[Definition 2.9]{DIRRT}.

Note the following easy example.

\begin{Ex}\label{Ex_left_perp}
For each $M \in \mod A$, 
we have a torsion pair $({^\perp M},\sfF(M))$,
so the torsion class ${^\perp M}$ is cocompact.
\end{Ex}

An additive full subcategory $\calC \subset \mod A$ is said to be 
\emph{functorially finite} if every $M \in \mod A$ admits
a left $\calC$-approximation and a right $\calC$-approximation.
We write $\ftors A$ for the set of functorially finite torsion classes.
The following property called the \emph{Smal{\o} symmetry} is our start point.
In particular, the ``if'' part of Conjecture \ref{Conj_bicpt} holds true.

\begin{Prop}\label{Prop_Smalo}\cite[Theorem]{Smalo}
Let $(\calT,\calF)$ be a torsion pair in $\mod A$.
Then the following conditions are equivalent.
\begin{enumerate2}
\item
The torsion class $\calT$ is functorially finite.
\item
The torsion-free class $\calF$ is functorially finite.
\item
There exists $M \in \mod A$ such that $\calT=\Fac M$.
\item
There exists $N \in \mod A$ such that $\calF=\Sub N$.
\end{enumerate2}
In this case, $\calT$ is bicompact.
\end{Prop}

We say that a torsion pair $(\calT,\calF)$ is \emph{functorially finite}
if both $\calT$ and $\calF$ are functorially finite.
It is equivalent to that $\calT$ or $\calF$ is functorially finite
by Proposition \ref{Prop_Smalo}.

As a notion related to functorially finite torsion classes,
we explain 2-term presiltng complexes introduced in \cite{KV}
in the perfect derived category $\sfK^\rmb(\proj A)$.

Let $U$ be a complex in $\sfK^\rmb(\proj A)$.
Then $U$ is said to be \emph{presilting} 
if $\Hom_{\sfK^\rmb(\proj A)}(U,U[>0])=0$.
A presilting complex $U$ is said to be \emph{silting} 
if $\sfK^\rmb(\proj A)$ itself
is the smallest thick subcategory of $\sfK^\rmb(\proj A)$ containing $U$.
Moreover $U$ is said to be \emph{2-term}
if $U$ is isomorphic to some complex of the form $U^{-1} \to U^0$
whose terms except the $-1$st and the $0$th ones vanish.

We write $\twopsilt A$ (resp.~$\twosilt A$)
for the set of isoclasses of basic 2-term presilting (resp.~silting)
complexes in $\sfK^\rmb(\proj A)$.
Then we have the following torsion pairs for each $U \in \twopsilt A$, 
which are functorially finite by Proposition \ref{Prop_Smalo}.

\begin{Def}\label{Def_ovT_U}
For each $U \in \twopsilt A$, we define two functorially finite torsion pairs
\begin{align*}
(\ovcalT_U,\calF_U)=({^\perp H^{-1}(\nu U)},\Sub H^{-1}(\nu U)), \quad
(\calT_U,\ovcalF_U)=(\Fac H^0(U),{H^0(U)^\perp}).
\end{align*}
\end{Def}

We obtain all functorially finite torsion classes in this way.

\begin{Prop}\label{Prop_AIR}\cite[Theorems 2.7, 3.2]{AIR}
There exist two surjections
$\twopsilt A \to \ftors A$ given by 
$U \mapsto \ovcalT_U$ and $U \mapsto \calT_U$.
Moreover we have a bijection $\twosilt A \to \ftors A$
given by $T \mapsto \ovcalT_T=\calT_T$.
\end{Prop}

To consider more torsion classes, 
we use the Grothendieck groups $K_0(\mod A)$ and $K_0(\proj A)$,
where $\proj A$ is seen as an exact category.

It is well-known that $K_0(\mod A)$ is naturally identified with
the Grothendieck group $K_0(\sfD^\rmb(\mod A))$
of the bounded derived category $\sfD^\rmb(\mod A)$
as a triangulated category.
Similarly, $K_0(\proj A)$ is identified with $K_0(\sfK^\rmb(\proj A))$.
These Grothendieck groups give rise to the real Grothendieck groups
$K_0(\proj A)_\R:=K_0(\proj A) \otimes_\Z \R$ and 
$K_0(\mod A)_\R:=K_0(\mod A) \otimes_\Z \R$.

They have the \emph{Euler bilinear form}
$\langle !,? \rangle \colon K_0(\proj A)_\R \times K_0(\mod A)_\R \to \R$
which satisfies
\begin{align*}
\langle [U], [X] \rangle
=\sum_{l \in \Z}\dim_K \Hom_{\sfD^\rmb(\mod A)}(U,X[l])
\end{align*}
for any $U \in \sfK^\rmb(\proj A)$ and $X \in \sfD^\rmb(\mod A)$.

The isoclasses $[S(1)],\ldots,[S(n)]$ of simple $A$-modules give 
a canonical basis of the Grothendieck group $K_0(\mod A)$;
that is, $K_0(\mod A)=\bigoplus_{i=1}^n \Z[S(i)]$.
Similarly, the isoclasses $[P(1)],\ldots,[P(n)]$ 
of indecomposable projective $A$-modules 
give a basis of the Grothendieck group $K_0(\proj A)$;
that is, $K_0(\proj A)=\bigoplus_{i=1}^n \Z[P(i)]$.
These bases are dual up to rescaling
with respect to the Euler bilinear form above;
more precisely,
\begin{align*}
\langle [P(i)],[S(j)] \rangle=\delta_{i,j} \dim_K \End_A (S(j)).
\end{align*}

Via the Euler bilinear form, each element $\theta \in K_0(\proj A)_\R$
is identified with the $\R$-linear form
$\theta \colon K_0(\mod A)_\R \to \R$ given by 
$v \mapsto \langle \theta,v \rangle$.
In particular, for $M \in \mod A$, we have a real number
$\theta(M):=\theta([M]) \in \R$.
Then we can consider the following class of torsion pairs.

\begin{Def}\label{Def_sst_tors}\cite[Subsection 3.1]{BKT}
Let $\theta \in K_0(\proj A)_\R$.
Then we define two \emph{semistable torsion pairs}
$(\ovcalT_\theta,\calF_\theta)$ and $(\calT_\theta,\ovcalF_\theta)$ by
\begin{align*}
\ovcalT_\theta&:=\{M \in \mod A \mid
\text{$\theta(N) \ge 0$ for any factor module $N$ of $M$}\}, \\
\calF_\theta&:=\{M \in \mod A \mid
\text{$\theta(L)<0$ for any nonzero submodule $L \ne 0$ of $M$}\}, \\
\calT_\theta&:=\{M \in \mod A \mid
\text{$\theta(N)>0$ for any nonzero factor module $N \ne 0$ of $M$}\}, \\
\ovcalF_\theta&:=\{M \in \mod A \mid
\text{$\theta(L) \le 0$ for any submodule $L$ of $M$}\}.
\end{align*}
\end{Def}

Note that $\calT_\theta \subset \ovcalT_\theta$ always holds.

We consider the following equivalence relation on $K_0(\proj A)_\R$.

\begin{Def}\label{Def_TF_eq}\cite[Definition 2.13]{A-wc}
We define an equivalence relation on $K_0(\proj A)_\R$
called the \emph{TF equivalence} as follows.
Two elements $\theta,\eta \in K_0(\proj A)_\R$ are said to be 
\emph{TF equivalent} if 
$\ovcalT_\theta=\ovcalT_\eta$ and $\calT_\theta=\calT_\eta$.
\end{Def}

Typical examples of TF equivalence classes are 
given by 2-term presilting complexes. 
For each $U=\bigoplus_{i=1}^m U_i \in \twopsilt A$ 
with each $U_i$ indecomposable,
we set 
\begin{align*}
C^\circ(U):=\sum_{i=1}^m \R_{>0}[U_i], \quad
C(U):=\sum_{i=1}^m \R_{\ge 0}[U_i].
\end{align*}
Then $C(U)$ is an $m$-dimensional polyhedral cone 
by \cite[Theorem 2.27, Corollary 2.28]{AiI},
and $C^\circ(U)$ is its relative interior.

\begin{Prop}\label{Prop_Yurikusa}
Let $U \in \twopsilt A$.
\begin{enumerate}
\item
\cite[Proposition 3.3]{Yurikusa} (cf.~\cite[Proposition 3.27]{BST})
Let $U \in \twopsilt A$ and $\theta \in C^\circ(U)$.
Then we have functorially finite torsion pairs
$(\ovcalT_\theta,\calF_\theta)=(\ovcalT_U,\calF_U)$ and 
$(\calT_\theta,\ovcalF_\theta)=(\calT_U,\ovcalF_U)$.
\item
\cite[Proposition 3.11]{A-wc}
The set $C^\circ(U)$ is a TF equivalence class. 
\end{enumerate}
\end{Prop}

This and Proposition \ref{Prop_AIR} imply that
every functorially finite class $\calT$ has $\theta \in K_0(\proj A)_\R$
such that $\calT=\calT_\theta=\ovcalT_\theta$.
More explicitly, take the unique $T \in \twosilt A$ with 
$\calT=\calT_T=\ovcalT_T$, and then any $\theta \in C^\circ(T)$
satisfies $\calT=\calT_\theta=\ovcalT_\theta$.

We say that $\theta \in K_0(\proj A)_\R$ is \emph{rigid}
if there exists $U \in \twopsilt A$ such that $\theta \in C^\circ(U)$.
Proposition \ref{Prop_Yurikusa} (1) immediately gives the following.

\begin{Prop}\label{Prop_rigid_func_fin}
Let $\theta \in K_0(\proj A)_\R$ be rigid.
Then $\ovcalT_\theta$ and $\calT_\theta$ are functorially finite.
\end{Prop}

\section{Bicompactness of semistable torsion classes}\label{Sec_semistable}

In this section, we first show that Conjecture \ref{Conj_bicpt}
holds for semistable torsion classes.

\begin{Thm}\label{Thm_sst_bicpt}
Let $A$ be a finite dimensional algebra over a field $K$.
For each $\theta \in K_0(\proj A)_\R$, 
the following conditions are equivalent.
\begin{enumerate2}
\item
The element $\theta$ is rigid.
\item
The torsion class $\calT_\theta$ is bicompact.
\item
The torsion class $\calT_\theta$ is compact.
\item
The torsion class $\calT_\theta$ is functorially finite.
\item
The torsion class $\ovcalT_\theta$ is bicompact.
\item
The torsion class $\ovcalT_\theta$ is cocompact.
\item
The torsion class $\ovcalT_\theta$ is functorially finite.
\end{enumerate2}
\end{Thm}

We mainly focus on showing (c)$\Rightarrow$(a),
since (a)$\Rightarrow$(d)$\Rightarrow$(b)$\Rightarrow$(c) 
easily follow from the previous section,
and since the remaining part is the dual.
To prove (c)$\Rightarrow$(a), 
the following result plays an important role.

\begin{Prop}\label{Prop_chamber}\cite[Theorem 3.17]{A-wc}
There exists a bijection
\begin{align*}
\twosilt A \to \{\emph{TF equivalence classes with nonempty interiors}\}
\end{align*}
given by $T \mapsto C^\circ(T)$.
\end{Prop}

Thus we immediately have the next property.

\begin{Lem}\label{Lem_E_TF}
Let $E$ be an open subset of $K_0(\proj A)_\R$.
If $E$ is contained in a single TF equivalence class,
then there exists $T \in \twosilt A$ 
such that $E \subset C^\circ(T)$.
\end{Lem}

We also need the elementary observations below.
For $\theta,\eta \in K_0(\proj A)_\R$,
we write $\eta<_{\cw}\theta$ if 
$\theta-\eta \in \sum_{i=1}^n \R_{>0}[P(i)] \subset K_0(\proj A)_\R$.

\begin{Lem}\label{Lem_open_convex}
The following statements hold.
\begin{enumerate}
\item
Let $M \in \mod A$.
Then $\{\eta \in K_0(\proj A)_\R \mid M \in \calT_\eta\}$
is open in $K_0(\proj A)_\R$.
\item
Let $\theta,\eta \in K_0(\proj A)_\R$.
If $\eta<_{\cw}\theta$, 
then we have $\calT_\eta \subset \ovcalT_\eta \subset \calT_\theta$.
\end{enumerate}
\end{Lem}

\begin{proof}
(1)
This set is equal to $\{\theta \in K_0(\proj A)_\R \mid 
\text{$\theta(v)>0$ for any $v \in F_M$}\}$,
where $F_M:=\{[N] \mid \text{$N$ is a nonzero factor module of $M$}\}
\subset K_0(\mod A)$.
Since $F_M$ is a finite set,
we have the assertion.

(2)
The inclusion $\ovcalT_\eta \subset \calT_\theta$ follows from
$\eta(M)<\theta(M)$ for any nonzero $M \in \mod A \setminus \{0\}$.
The other inclusion $\calT_\eta \subset \ovcalT_\eta$ is obvious.
\end{proof}

Then we have the following crucial statement.

\begin{Prop}\label{Prop_T_cpt_rigid}
Let $\theta \in K_0(\proj A)_\R$.
If $\calT_\theta$ is compact, then $\theta$ is rigid.
\end{Prop}

\begin{proof}
Take $M \in \mod A$ such that $\calT_\theta=\sfT(M)$.
We set 
\begin{align*}
E_1:=\{\eta \in K_0(\proj A)_\R \mid M \in \calT_\eta\}, \quad
E_2:=\{\eta \in K_0(\proj A)_\R \mid \eta<_{\cw}\theta\}, \quad
E:=E_1 \cap E_2.
\end{align*}

By Lemma \ref{Lem_open_convex} (1), $E_1$ is open.
Since $\theta \in E_1$, 
we can take $\epsilon \in \R_{>0}$ such that
$\theta-\epsilon \sum_{i=1}^n [P(i)] \in E_1$.
This element belongs also to $E_2$; hence $E$ is nonempty. 

Clearly, $E_2$ is open.
Thus by Lemma \ref{Lem_open_convex} (1) again, $E$ is also open.

Let $\eta \in E$.
Since $\eta<_{\cw}\theta$, 
Lemma \ref{Lem_open_convex} (2) implies 
$\calT_\eta \subset \ovcalT_\eta \subset \calT_\theta$.
Then $\calT_\theta=\sfT(M)$ and $M \in \calT_\eta$ give
$\calT_\eta=\ovcalT_\eta=\calT_\theta$.
In particular, all elements in $E$ are TF equivalent.

Thus the open nonempty set $E$ is contained in a single TF equivalence class.
Then by Lemma \ref{Lem_E_TF},
we take $T \in \twosilt A$ such that $E \subset C^\circ(T)$.

We claim $\theta \in \overline{E}$.
Since $\calT_\theta=\sfT(M)$, we get $\theta \in E_1$,
so $E_1$ is an open neighborhood of $\theta$.
Clearly, $\theta \in \overline{E_2}$ holds.
These imply $\theta \in \overline{E_1 \cap E_2}=\overline{E}$.

Therefore we get $\theta \in \overline{E}=C(T)$.
By definition, there exists a direct summand $U$ of $T$ 
such that $\theta \in C^\circ(U)$.
Thus $\theta$ is rigid.
\end{proof}

We remark that $U$ is the Bongartz cocompletion of $T$,
since $\calT_T=\calT_\eta=\calT_\theta=\calT_U$.

Now we can show Theorem \ref{Thm_sst_bicpt}.

\begin{proof}[Proof of Theorem \ref{Thm_sst_bicpt}]
We only show the equivalence of (a), (b), (c) and (d).
The equivalence of (a), (e), (f) and (g) is dual.

(a)$\Rightarrow$(d) is Proposition \ref{Prop_rigid_func_fin},
(d)$\Rightarrow$(b) is Proposition \ref{Prop_Smalo}, and
(b)$\Rightarrow$(c) is obvious.
Finally,
(c)$\Rightarrow$(a) follows from Proposition \ref{Prop_T_cpt_rigid}.
\end{proof}

In the rest of this section, 
we assume that $A$ is a finite dimensional algebra over 
an algebraically closed field $K$,
and only consider lattice points $\theta \in K_0(\proj A)$.
We will show the following.

\begin{Thm}\label{Thm_ovT_Z_bicpt}
Let $A$ be a finite dimensional algebra 
over an algebraically closed field $K$.
For each $\theta \in K_0(\proj A)_\R$, 
each of the following condition is equivalent to 
(a)--(g) in Theorem \ref{Thm_sst_bicpt}.
\begin{enumerate2}
\item[\textup{(h)}]
The torsion class $\calT_\theta$ is cocompact.
\item[\textup{(i)}]
The torsion class $\ovcalT_\theta$ is compact.
\end{enumerate2}
\end{Thm}

Let $\theta \in K_0(\proj A)$. 
Take $P_0^\theta,P_1^\theta \in \proj A$
such that $\theta=[P_0^\theta]-[P_1^\theta]$ and 
$\add P_0^\theta \cap \add P_1^\theta=\{0\}$.
Note that such $P_0^\theta,P_1^\theta \in \proj A$ are unique
up to isomorphisms.
Then we set $\Hom(\theta):=\Hom_A(P_1^\theta,P_0^\theta)$.
For each $f \in \Hom(\theta)$, 
we define $\ovcalT_f:={^\perp (\Ker \nu f)} \in \tors A$,
where $\nu$ is the Nakayama functor $\proj A \to \inj A$.

The following crucial property comes from geometric invariant theory
of quiver representations.

\begin{Prop}\label{Prop_ovT_f} \cite[Theorem 4.3]{AsI} 
(cf.~\cite[Theorem 3.6]{Fei}) 
Let $\theta \in K_0(\proj A)$.
Then we have
\begin{align*}
\ovcalT_\theta=\bigcup_{f \in \Hom(l\theta), \ l \in \Z_{\ge 1}}\ovcalT_f.
\end{align*}
\end{Prop}

Then we obtain the following property.

\begin{Prop}\label{Prop_ovT_Z_cpt_cocpt}
Let $\theta \in K_0(\proj A)$.
If $\ovcalT_\theta$ is compact, then $\ovcalT_\theta$ is cocompact.
\end{Prop}

\begin{proof}
Take $M \in \mod A$ such that $\ovcalT_\theta=\sfT(M)$.
Since $M \in \ovcalT_\theta$, by Proposition \ref{Prop_ovT_f}, 
we take $l \in \Z_{\ge 1}$ and $f \in \Hom(l\theta)$
such that $M \in \ovcalT_f$.
Thus we get $\ovcalT_\theta=\sfT(M) \subset \ovcalT_f$.
This and Proposition \ref{Prop_ovT_f} give $\ovcalT_\theta=\ovcalT_f$,
so $\ovcalT_\theta$ is cocompact by Example \ref{Ex_left_perp}.
\end{proof}

Now we can show Theorem \ref{Thm_ovT_Z_bicpt}.

\begin{proof}[Proof of Theorem \ref{Thm_ovT_Z_bicpt}]
(i)$\Rightarrow$(f) follows from Proposition \ref{Prop_ovT_Z_cpt_cocpt},
and (e)$\Rightarrow$(i) is obvious.
Thus, (i) is equivalent to the conditions in Theorem \ref{Thm_sst_bicpt}.
Dually, so is (h).
\end{proof}

\section{Numerically disjoint torsion classes}
\label{Sec_hered}

In this section, we take more direct approaches 
to Conjecture \ref{Conj_bicpt}.
For this purpose, we consider the conditions below. 

\begin{Def}
We define the following notions.
\begin{enumerate}
\item
Let $\calC,\calD \subset \mod A$ be additive subcategories.
Then we say that $\calC$ and $\calD$ are \emph{numerically disjoint}
if there exist no $X \in \calC$ and $Y \in \calD$
such that $[X]=[Y] \ne 0$ in $K_0(\mod A)$.
\item
Let $\calT \in \tors A$.
Then $\calT$ is called a \emph{numerically disjoint} torsion class
if $\calT$ and $\calT^\perp$ are numerically disjoint.
\end{enumerate}
\end{Def}

We immediately have the following.

\begin{Lem}\label{Lem_num_dis}
Let $\theta \in K_0(\proj A)_\R$.
Then $\calT_\theta \in \tors A$ is numerically disjoint.
\end{Lem}

\begin{proof}
If $X \in \calT_\theta$ and $Y \in \ovcalF_\theta$ are nonzero, 
then we have $\theta(X)>0 \ge \theta(Y)$, which imply $[X] \ne [Y]$.
\end{proof}

We aim to show the following result.

\begin{Thm}\label{Thm_bicpt_num_dis}
Let $A$ be a finite dimensional algebra over a field $K$, 
and $\calT \in \tors A$.
Then $\calT$ is functorially finite if and only if
$\calT$ is bicompact and numerically disjoint.
\end{Thm}

To prove this, we introduce another condition on subcategories.
For each subset $S \subset K_0(\mod A)_\R$,
we define $\cone S$ as the cone in $K_0(\mod A)_\R$ generated by $S$.

\begin{Def}\label{Def_C_F_C_S}
Let $\calC \subset \mod A$ be an additive full subcategory.
\begin{enumerate}
\item
We set $\cone \calC:=\cone \{[X] \mid X \in \calC\} \subset K_0(\mod A)_\R$.
\item
We say that $\calC$ is \emph{polyhedral} 
if $\cone \calC$ is a polyhedral cone.
\end{enumerate}
\end{Def}

The following observation is useful.

\begin{Lem}\label{Lem_cpt_poly}
If $\calT \in \tors A$ is compact, then $\calT$ is polyhedral.
\end{Lem}

\begin{proof}
Take $M \in \mod A$ such that $\calT=\sfT(M)$.
Then every module in $\calT$ is filtered by factor modules of $M$.
Since $\{[M] \mid \text{$M$ is a factor module}\} \subset K_0(\mod A)$ 
is a finite set,
$\calT=\sfT(M)$ is polyhedral.
\end{proof}

Then we have the next characterizations.

\begin{Lem}\label{Lem_C_F_C_S}
Let $\calT \in \tors A$ and $\calF \in \torf A$.
Then the following conditions are equivalent.
\begin{enumerate2}
\item
The subcategories $\calT$ and $\calF$ are numerically disjoint.
\item
We have $\cone \calT \cap \cone \calF=\{0\}$.
\item
The cone $C:=\{x-y \mid x \in \cone \calT, \ y \in \cone \calF\}$
is strongly convex.
\end{enumerate2}
If $\calT$ and $\calF$ are polyhedral, 
then (a)--(c) are equivalent to
the condition below.
\begin{enumerate2}
\item[\textup{(d)}]
There exists $\theta \in K_0(\proj A)_\R$ such that
$\calT \subset \calT_\theta$ and $\calF \subset \calF_\theta$.
\end{enumerate2}
\end{Lem}

\begin{proof}
(a)$\Rightarrow$(b)
Assume the contrary $\cone \calT \cap \cone \calF \ne \{0\}$ of (b).
Then there exist $X_1,\ldots,X_p \in \calT$, $Y_1,\ldots,Y_q \in \calF$
and positive integers $a_1,\ldots,a_p,b_1,\ldots,b_q \in \Z_{\ge 1}$
such that
\begin{align*}
\sum_{i=1}^p a_i[X_i]=\sum_{j=1}^q b_j[Y_j] \ne 0 \in K_0(\mod A).
\end{align*}
Set $X:=\bigoplus_{i=1}^p X_i^{\oplus a_i}$ and 
$Y:=\bigoplus_{j=1}^q Y_j^{\oplus b_j}$.
Then we get $X \in \calT$, $Y \in \calF$ and $[X]=[Y] \ne 0$.
Thus (a) is denied.

(b)$\Rightarrow$(c) and (c)$\Rightarrow$(a) are easy.
Moreover (d)$\Rightarrow$(a) is Lemma \ref{Lem_num_dis}.

We additionally assume that $\calT,\calF$ are polyhedral.

(c)$\Rightarrow$(d)
Since $\calT$ and $\calF$ are polyhedral,
$C$ is a polyhedral cone.
This and the assumption that $C$ is strongly convex allow us 
to take $\theta \in K_0(\proj A)_\R$
such that any nonzero $v \in C \setminus \{0\}$ satisfies $\theta(v)>0$.
Let $X \in \calT$. For any factor module $X'$ of $X$,
we have $\theta(X')>0$, since $X' \in \calT$ gives $[X'] \in C$.
Thus $X \in \calT_\theta$ holds.
Therefore we obtain $\calT \subset \calT_\theta$.
Dually, $\calF \subset \calF_\theta$ follows.
\end{proof}

Now we are able to show Theorem \ref{Thm_bicpt_num_dis}.

\begin{proof}[Proof of Theorem \ref{Thm_bicpt_num_dis}]
The ``only if'' part is obvious 
by Theorem \ref{Thm_sst_bicpt} and Lemma \ref{Lem_num_dis}.

It remains to show the ``if'' part.
Since $\calT$ is bicompact,
Lemma \ref{Lem_cpt_poly} and its dual imply that
both $\calT$ and $\calT^\perp$ are polyhedral.
Then Lemma \ref{Lem_C_F_C_S} (a)$\Rightarrow$(d) allows us to take 
$\theta \in K_0(\proj A)_\R$
such that $\calT \subset \calT_\theta$ and 
$\calF \subset \calF_\theta \subset \ovcalF_\theta$.
Since $(\calT,\calT^\perp)$ and $(\calT_\theta,\ovcalF_\theta)$ 
are torsion pairs in $\mod A$,
we have $(\calT,\calT^\perp)=(\calT_\theta,\ovcalF_\theta)$.
Thus $\calT=\calT_\theta$ is bicompact, 
so Theorem \ref{Thm_sst_bicpt} (b)$\Rightarrow$(d) implies that
$\calT=\calT_\theta$ is functorially finite.
\end{proof}

In the rest of this section, let $K$ be an algebraically closed.
Then we may assume that a finite dimensional algebra $A$
is of the form $KQ/I$ with $Q=(Q_0,Q_1)$ a finite quiver 
and $I \subset KQ$ an admissible ideal.
We briefly recall basic notions on representation varieties.
See \cite{CB,DW} for details.

Let $d=(d_i)_{i \in Q_0} \in (\Z_{\ge 0})^{Q_0}$ be a dimension vector.
It is identified with the element 
$\sum_{i \in Q_0}d_i[S(i)]$ of $K_0(\mod A)_{\ge 0}$.
We define $\rep(Q,d)$ as the direct product 
\begin{align*}
\prod_{(\alpha \colon i \to j) \in Q_1} \Hom_K(K^{d_i},K^{d_j})
\end{align*}
equipped with the Zariski topology.
Moreover $\rep(A,d)$ denotes the closed subset of $\rep(Q,d)$
given by the relations corresponding to $I$.
We call $\rep(Q,d)$ and $\rep(A,d)$ the \emph{representation varieties}
of $Q$ and $A$ for the dimension vector $d$. 

While $\rep(Q,d)$ is always irreducible, $\rep(A,d)$ may not be irreducible.
We write $\Irr(A,d)$ for the set of irreducible components of $\rep(A,d)$.

Let $\calZ \in \Irr(A,d)$.
For a full subcategory $\calC \subset \mod A$,
we define $\calZ \cap \calC:=\{X \in \calZ \mid X \in \calC\}$.
Since the functions
$\dim_K \Hom(?,M) \colon \calZ \to \Z_{\ge 0}$ and 
$\dim_K \Hom(M,?) \colon \calZ \to \Z_{\ge 0}$
are upper-semicontinuous,
we obtain the following.

\begin{Lem}\label{Lem_Hom_open}
Let $A=KQ/I$ as above, $d \in (\Z_{\ge 0})^{Q_0}$ and $\calZ \in \Irr(A,d)$.
For each $M \in \mod A$, both
$\calZ \cap {^\perp M}$ and $\calZ \cap {M^\perp}$ are 
open subsets of $\calZ$.
\end{Lem}

This gives the following.

\begin{Prop}\label{Prop_either_empty}
Let $A=KQ/I$ as above, $d \in (\Z_{\ge 0})^{Q_0}$ and $\calZ \in \Irr(A,d)$.
If $\calT \in \tors A$ is bicompact and $d \ne 0$,
then at least one of $\calZ \cap \calT$ and $\calZ \cap \calT^\perp$ is empty.
\end{Prop}

\begin{proof}
Since $\calT$ is bicompact, we take $M,N \in \mod A$
such that $(\calT,\calT^\perp)=(\sfT(M),\sfF(N))$.
Then $(\calT,\calT^\perp)=({^\perp N},M^\perp)$ holds.
Thus we get $\calZ \cap \calT=\calZ \cap {^\perp N}$ and 
$\calZ \cap \calT^\perp=\calZ \cap M^\perp$.

Now assume that both of them are nonempty.
Then they are open dense in $\calZ$ by Lemma \ref{Lem_Hom_open}.
Since $\calZ$ is irreducible,
the intersection $\calZ \cap \calT \cap \calT^\perp=\calZ \cap \{0\}$ 
is nonempty.
This implies $d=0$, a contradiction.
\end{proof}

Thus a pontential counterexample of Conjecture \ref{Conj_bicpt} 
must satisfy the following property.

\begin{Prop}\label{Prop_dis_irr_comp}
Let $A=KQ/I$ as above.
Assume that $\calT \in \tors A$ is bicompact but not functorially finite.
Then there exists $d \in (\Z_{\ge 0})^{Q_0}$ such that 
$\rep(A,d)$ admits distinct irreducible components 
$\calZ_1 \ne \calZ_2 \in \Irr(A,d)$
such that $\calZ_1 \cap \calT$ and $\calZ_2 \cap \calT^\perp$ 
are both nonempty.
\end{Prop}

\begin{proof}
By Theorem \ref{Thm_bicpt_num_dis}, 
$\calT$ cannot be numerically disjoint.
Thus there exist $X \in \calT$ and $Y \in \calT^\perp$ such that
$d:=[X]=[Y] \ne 0$.
We take irreducible components $\calZ_1,\calZ_2 \in \Irr(A,d)$
such that $X \in \calZ_1$ and $Y \in \calZ_2$.
Then $\calZ_1 \cap \calT$ and $\calZ_2 \cap \calT^\perp$ are both nonempty,
so Proposition \ref{Prop_either_empty} implies $\calZ_1 \ne \calZ_2$.
\end{proof}

Then we immediately have the following main result for hereditary algebras.

\begin{Thm}\label{Thm_hered}
Assume that $A$ is a finite dimensional hereditary algebra 
over an algebraically closed field $K$.
Then Conjecture \ref{Conj_bicpt} holds;
that is, $\calT \in \tors A$ is functorially finite
if and only if $\calT$ is bicompact.
\end{Thm}

\begin{proof}
We may assume that $A=KQ$ with $Q$ a finite acyclic quiver.
Then $\rep(A,d)=\rep(Q,d)$ is irreducible for any $d \in (\Z_{\ge 0})^{Q_0}$.
Thus Proposition \ref{Prop_dis_irr_comp} 
implies Conjecture \ref{Conj_bicpt}.
\end{proof}

We remark that Theorem \ref{Thm_ovT_Z_bicpt} is not used 
to prove Theorem \ref{Thm_hered}.
Though both theorems have the assumption that $K$ is algebraically closed,
they are not directly related.
It is an interesting problem whether Theorem \ref{Thm_hered} can
be extended to hereditary algebras over any base field.

\section{Conjectures on brick infiniteness}\label{Sec_Conj}

In this section, we apply our results so far
to two conjectures by Enomoto \cite{Enomoto} 
and Demonet \cite{Demonet} on brick infiniteness.
We here recall some terminology.

A module $M \in \mod A$ is called a \emph{brick}
if $\End_A(M)$ is a division $K$-algebra
(it is equivalent to $\End_A(M) \simeq K$ if $K$ is algebraically closed),
and a \emph{semibrick} $\calS$ is the set of isoclasses of bricks
such that $\Hom_A(M,N)=0$ for any $M \ne N \in \calS$.
A semibrick is said to be \emph{finite} (resp.~\emph{infinite})
if it is a finite (resp.~infinite) set of isoclasses of bricks.

We say that $A$ is \emph{brick finite} 
if there are only finitely many isoclasses of bricks in $\mod A$,
and otherwise, $A$ is said to be \emph{brick infinite}.

By definition, if there exists an infinite semibrick in $\mod A$,
then $A$ is brick infinite.
Enomoto proposed that the converse also holds.

\begin{Conj}[Enomoto's Infinite Semibrick Conjecture]
\label{Conj_Enomoto}\cite[Conjecture 5.12]{Enomoto}
Let $A$ be a finite dimensional algebra over a field $K$. 
If $A$ is brick infinite, then there exists an infinite semibrick in $\mod A$.
\end{Conj}

This is one of the two conjectures we consider.
The other conjecture by Demonet is based 
on the following known characterizations of brick infinite algebras.

\begin{Prop}\label{Prop_brick_infin}
\cite[Theorem 4.2]{DIJ} \cite[Theorem 4.7]{A-wc}
For a finite dimensional algebra $A$ over a field $K$, 
the following conditions are equivalent.
\begin{enumerate2}
\item
The algebra $A$ is brick infinite.
\item
The set $\twosilt A$ is infinite.
\item
There exists $\theta \in K_0(\proj A)_\R$ which is not rigid.
\end{enumerate2}
\end{Prop}

Strengthening this, Demonet suggested the following conjecture.

\begin{Conj}[Demonet's Lattice Point Conjecture]
\label{Conj_Demonet}\cite[Question 3.49]{Demonet}
Let $A$ be a finite dimensional algebra over a field $K$. 
If $A$ is brick infinite, then there exists a lattice point 
$\theta \in K_0(\proj A)$ which is not rigid.
\end{Conj}

We prove that these two conjectures have the following relationship.

\begin{Thm}\label{Thm_Demonet_Enomoto}
Assume that $A$ is a finite dimensional algebra 
over an algebraically closed field $K$.
\begin{enumerate}
\item
Let $\theta \in K_0(\proj A)$ be a lattice point.
If $\theta$ is not rigid, 
then there exists an infinite semibrick $\calS$ 
such that $\ovcalT_\theta=\sfT(\calS)$.
\item
If $A$ satisfies Demonet's Lattice Point Conjecture \ref{Conj_Demonet},
then $A$ satisfies Enomoto's Infinite Semibrick Conjecture \ref{Conj_Enomoto}.
\end{enumerate}
\end{Thm}

Before proceeding to the proof, we explain some relationship
between semibricks and torsion classes.

By \cite[1.2]{Ringel1}, there are bijections between
the set $\sbrick A$ of semibricks in $\mod A$ and
the set $\wide A$ of \emph{wide subcategories} of $\mod A$,
which are defined as abelian subcategories closed 
under extensions in $\mod A$.
Each $\calS \in \sbrick A$ corresponds to $\Filt \calS \in \wide A$,
and $\calS$ is the set of simple objects of $\Filt \calS$.

As in \cite{AP},
$\calT \in \tors A$ is said to be \emph{widely generated}
if there exists $\calW \in \wide A$ such that $\calT=\sfT(\calW)$;
or equivalently, there exists $\calS \in \sbrick A$ 
such that $\calT=\sfT(\calS)$.
For each $\calT \in \tors A$,
such $\calW$, or equivalently $\calS$, 
is unique if exists by \cite[Proposition 3.3]{MS}.

To show Theorem \ref{Thm_Demonet_Enomoto}, we use the following property.

\begin{Prop}\label{Prop_ovT_wide_gen}\cite[Theorem 2.23]{A-wc}
Let $\theta \in K_0(\proj A)$ be a lattice point.
Then the torsion class $\ovcalT_\theta$ is widely generated.
\end{Prop}

Then Theorem \ref{Thm_Demonet_Enomoto} is quickly shown as follows.

\begin{proof}[Proof of Theorem \ref{Thm_Demonet_Enomoto}]
(1)
By Proposition \ref{Prop_ovT_wide_gen},
we take a semibrick $\calS$ such that $\ovcalT_\theta=\sfT(\calS)$.
Assume that $\calS$ is finite. 
Writing $\calS=\{M_1,\ldots,M_m\}$,
clearly $\ovcalT_\theta=\sfT(\bigoplus_{i=1}^m M_i)$ is compact.
Since $K$ is algebraically closed and 
$\theta \in K_0(\proj A)$ is a lattice point,
Theorem \ref{Thm_ovT_Z_bicpt} implies that $\theta$ is rigid,
a contradiction.
Thus $\calS$ cannot be finite.

(2) is immediate from (1).
\end{proof}

As in this proof, the torsion class $\sfT(\calS)$ generated 
by a finite semibrick $\calS$ is always compact.
Actually, all compact torsion classes are obtained in this way.

\begin{Prop}\label{Prop_cpt_fin_sbrick}\cite[Proposition 5.5]{Ringel2}
Let $\calT \in \tors A$.
Then $\calT$ is compact if and only if
there exists a finite semibrick $\calS$ such that $\calT=\sfT(\calS)$.
\end{Prop}

We give more observations 
on brick (in)finiteness,
combining existing characterizations with our results.
As the dual of widely generated torsion classes,
we say that $\calT \in \tors A$ is \emph{widely cogenerated}
if there exists $\calS \in \sbrick A$ such that
$\calT^\perp=\sfF(\calS)$.
As a consequence of the papers \cite{AS,DIJ,Ringel2,Sentieri1}, 
we have the following characterization of brick finite algebras.

\begin{Cor}\label{Cor_brick_infin_tors}
For a finite dimensional algebra $A$, the following conditions are equivalent.
\begin{enumerate2}
\item
The algebra $A$ is brick finite.
\item
The set $\tors A$ is finite.
\item
Every $\calT \in \tors A$ is functorially finite.
\item
Every $\calT \in \tors A$ is bicompact.
\item
Every $\calT \in \tors A$ is compact.
\item
Every $\calT \in \tors A$ is cocompact.
\item
Every $\calT \in \tors A$ is widely generated.
\item
Every $\calT \in \tors A$ is widely cogenerated.
\end{enumerate2}
\end{Cor}

\begin{proof}
(a), (b) and (c) are equivalent 
by \cite[Corollary 2.9, Theorems 3.8, 4.2]{DIJ}.
(c)$\Rightarrow$(d) follows from Proposition \ref{Prop_Smalo}.
(d)$\Rightarrow$(e) is clear.
(e)$\Rightarrow$(g) follows from Proposition \ref{Prop_cpt_fin_sbrick}.
(g)$\Rightarrow$(a) follows from \cite[Proposition 5.4]{AS}
based on \cite[Corollary 4.10]{Sentieri1}.
The remaining parts are the duals.
\end{proof}

Actually, our results give that it is enough to consider
semistable torsion classes in (c), (d), (e) and (f) 
in Corollary \ref{Cor_brick_infin_tors}.
We define some subsets of $\tors A$ by
\begin{align*}
\tors_\R A&:=\{\calT_\theta \mid \theta \in K_0(\proj A)_\R\},&
\ovtors_\R A&:=\{\ovcalT_\theta \mid \theta \in K_0(\proj A)_\R\},\\
\tors_\Z A&:=\{ \calT_\theta \mid \theta \in K_0(\proj A)\},&
\ovtors_\Z A&:=\{ \ovcalT_\theta \mid \theta \in K_0(\proj A)\}.
\end{align*}
Then Theorem \ref{Thm_sst_bicpt} 
together with existing results give the following properties.

\begin{Cor}\label{Cor_brick_infin_tors_sst}
For a finite dimensional algebra $A$, 
the following conditions are equivalent.
\begin{enumerate2}
\item
The algebra $A$ is brick finite.
\item
Every $\calT \in \tors_\R A$ is functorially finite.
\item
Every $\calT \in \tors_\R A$ is bicompact.
\item
Every $\calT \in \tors_\R A$ is compact.
\item
Every $\calT \in \ovtors_\R A$ is functorially finite.
\item
Every $\calT \in \ovtors_\R A$ is bicompact.
\item
Every $\calT \in \ovtors_\R A$ is cocompact.
\end{enumerate2}
\end{Cor}

\begin{proof}
(a)$\Rightarrow$(b) is 
Corollary \ref{Cor_brick_infin_tors},
(b)$\Rightarrow$(c) is Proposition \ref{Prop_Smalo}, and
(c)$\Rightarrow$(d) is clear.
If (d) holds, then every $\theta \in K_0(\proj A)_\R$ is rigid
by Theorem \ref{Thm_sst_bicpt},
so (a) follows from Proposition \ref{Prop_brick_infin}.
The remaining part is dually shown.
\end{proof}

For algebras satisfying Conjecture \ref{Conj_Demonet},
we can say more.

\begin{Cor}
Assume that $A$ satisfies 
Demonet's Lattice Point Conjecture \ref{Conj_Demonet}.
\begin{enumerate}
\item
Corollary \ref{Cor_brick_infin_tors_sst} holds even if  
$\tors_\R A$ and $\ovtors_\R A$ in the conditions (b)--(g) 
are replaced with $\tors_\Z A$ and $\ovtors_\Z A$, respectively.
\item
Let the base field $K$ be algebraically closed.
Then (a)--(g) in Corollary \ref{Cor_brick_infin_tors_sst} 
are equivalent also to the following conditions.
\begin{enumerate2}
\item[\textup{(h)}]
Every $\calT \in \tors_\R A$ is cocompact.
\item[\textup{(h)$'$}]
Every $\calT \in \tors_\Z A$ is cocompact.
\item[\textup{(i)}]
Every $\calT \in \ovtors_\R A$ is compact.
\item[\textup{(i)$'$}]
Every $\calT \in \ovtors_\Z A$ is compact.
\end{enumerate2}
\end{enumerate}
\end{Cor}

\begin{proof}
(1)
Use Conjecture \ref{Conj_Demonet} 
instead of Proposition \ref{Prop_brick_infin} 
in the proof of Corollary \ref{Cor_brick_infin_tors_sst}.
 
(2)
(c)$\Rightarrow$(h)$\Rightarrow$(h)$'$ are obvious.
Consider the condition (c)$'$ every $\calT \in \tors_\Z A$ is bicompact.
Since $K$ is algebraically closed, (h)$'$$\Rightarrow$(c)$'$ follows from 
Theorem \ref{Thm_ovT_Z_bicpt}.
Then (1) implies (c)$'$$\Rightarrow$(c).
Thus (c) and (h) are equivalent.
Dually, so are (f) and (i). 
\end{proof}

On the other hand, our results cannot be used
to reduce the torsion classes considered in 
(g) and (h) from Corollary \ref{Cor_brick_infin_tors}.
Thus we pose the following question.
We remark that 
the ``$\Z$'' versions of (c) and (e) below do not appear, 
because they hold for all finite dimensional algebras 
by Proposition \ref{Prop_ovT_wide_gen}.

\begin{Ques}
For a finite dimensional algebra $A$, 
are the following conditions equivalent?
\begin{enumerate2}
\item
The algebra $A$ is brick finite.
\item
Every $\calT \in \tors_\R A$ is widely generated.
\item[\textup{(b)$'$}]
Every $\calT \in \tors_\Z A$ is widely generated.
\item
Every $\calT \in \ovtors_\R A$ is widely generated.
\item
Every $\calT \in \ovtors_\R A$ is widely cogenerated.
\item[\textup{(d)$'$}]
Every $\calT \in \ovtors_\Z A$ is widely cogenerated.
\item
Every $\calT \in \tors_\R A$ is widely cogenerated.
\end{enumerate2}
\end{Ques}

We expect that some results of this paper can be generalized to the setting of
\emph{algebraic} categories, 
that is, abelian length categories 
which has only finitely many isoclasses of simple objects.
We would like to discuss this in a future project.
For existing results under this setting, 
see \cite{BPPW,Sentieri2} for example.

\end{document}